\documentclass[11pt]{amsart}

\usepackage{amscd}
\usepackage{amsmath, amssymb}
\usepackage{amsfonts}

\newcommand{\ov}[1]{\overline{#1}}

\newcommand{\ve}{\varepsilon}

\numberwithin{equation}{section}

\renewcommand{\leq}{\leqslant}

\renewcommand{\le}{\leqslant}
\renewcommand{\ge}{\geqslant}

\begin{document}
\newtheorem{claim}{Claim}
\newtheorem{theorem}{Theorem}[section]
\newtheorem{conjecture}[theorem]{Conjecture}
\newtheorem{lemma}[theorem]{Lemma}
\newtheorem{corollary}[theorem]{Corollary}
\newtheorem{proposition}[theorem]{Proposition}
\newtheorem{question}[theorem]{question}
\newtheorem{defn}[theorem]{Definition}
\newtheorem{remark}{Remark}[section]

\author[A. Chau]{Albert Chau}
\author[B. Weinkove]{Ben Weinkove}

\address{Department of Mathematics \\ The University of British Columbia \\ 1984 Mathematics Road \\ Vancouver, B.C.,  Canada V6T 1Z2} 
\address{Department of Mathematics \\ Northwestern University \\ 2033 Sheridan Road \\ Evanston, IL 60208, USA}

\newenvironment{example}[1][Example]{\addtocounter{remark}{1} \begin{trivlist}
\item[\hskip
\labelsep {\bfseries #1  \thesection.\theremark}]}{\end{trivlist}}
\title[The second boundary value problem]{Monge-Amp\`ere functionals and \\ the second boundary value problem}

\thanks{Research supported in part by NSERC grant 327637-06 and NSF grant  DMS-1332196.  
Part of this work was carried out while the  first-named author was on sabbatical at the Department of Mathematics, Northwestern University and he thanks the department  for their kind hospitality.}

\begin{abstract}
We consider a  Monge-Amp\`ere functional and its corresponding second boundary value problem, a nonlinear fourth order PDE with two Dirichlet boundary conditions.  This problem was solved by Trudinger-Wang and Le under the assumption that the right hand side of the equation is nonpositive.  We remove this assumption, to settle the case of the second boundary value problem with arbitrary right hand side, in dimensions $n \ge 2$.
   In particular, this shows that one can prescribe the affine mean curvature of the graph of a convex function with Dirichlet boundary conditions on the function and the determinant of its Hessian.

We relate our results, and the case of $n=1$, to a notion of properness  for a certain functional on the set of convex functions. \end{abstract}

\maketitle

\section{Introduction}

Let $\Omega \subset \mathbb{R}^n$ be a uniformly convex  domain and $f$ a given function on $\Omega$.  For $u$ a strictly convex function on $\Omega$, 
we consider the Monge-Amp\`ere functional
\begin{equation}\label{functional}
u \mapsto \int_{\Omega}  G(d)dx -\int_{\Omega}  u fdx
\end{equation}
where $d= \det D^2u$ and $G$ is the concave function
\begin{equation} \label{G}
G(d) = \frac{d^{\theta}}{\theta}, \quad \textrm{for some } \theta \in [0,1/n),
\end{equation}
where we take the case $\theta=0$ to mean $G(d)=\log d$.

Write 
\begin{equation} \label{w}
w(d) := G'(d) = \frac{1}{d^{1-\theta}}>0.
\end{equation}
The Euler-Lagrange equation of 
$(\ref{functional})$, with respect to compactly supported perturbations, is the fourth order equation
\begin{equation} \label{Lu}
L[u]:=  U^{ij} (w(d))_{ij} = f \quad \textrm{on } \Omega,
\end{equation}
where we write $U^{ij}$ for the cofactor matrix of $(u_{ij})= D^2u$.

The functional (\ref{functional}) and the equation (\ref{Lu}) have appeared many times in the literature.  When $\theta = 1/(n+2)$ and $f=0$ the functional \eqref{functional} coincides with the affine area of the graph of $u$ and the quantity $L[u]$ is the affine mean curvature of the graph of $u$ (see \cite{TW00}, for example).  Trudinger-Wang \cite{TW05, TW08} solved the \emph{first boundary value problem} 
$$L[u]=f \quad \textrm{on } \Omega, \quad u=\varphi \quad \textrm{on } \partial \Omega, \quad Du \subseteq D\varphi \quad \textrm{on } \partial \Omega,$$
for a given uniformly convex $\varphi$ on $\Omega$.  The case of $f=0$ is  the affine Plateau problem for graphs.  

When $G(d)=\log d$ then $L[u]=f$ is known as Abreu's equation \cite{Ab}, arising in the work of Donaldson and others in the study of constant scalar curvature K\"ahler metrics on toric varieties (see \cite{D0, D1, D09, ZZ, FS, CHLS} for example).  
The functional corresponding to this problem is of the form
\begin{equation} \label{mabuchi}
u \mapsto \int_{\Omega} G(d)dx - \int_{\partial \Omega} u d\sigma + \int_{\Omega} u dA,
\end{equation}
for suitable measures $d\sigma$ and $dA$ on $\partial \Omega$ and $\Omega$ respectively.  Up to a sign, this functional is known as the Mabuchi energy.

Existence of solutions to the first boundary value problem for Abreu's equation was shown by Zhou \cite{Zhou}.

Motivated by Donaldson's work in complex geometry, Le-Savin \cite{LS, LS2} investigated the problem of maximizing the functional (\ref{mabuchi})
for quite general $G$, $dA$ and $d\sigma$.  They studied the corresponding Euler-Lagrange equation (\ref{Lu}) with two boundary conditions for $w$.  Their solution of the problem required a stability condition, in the sense of Donaldson, on the linear part of this functional.

In \cite{D2}, Donaldson  investigated the functional (\ref{functional}), with $f=0$, for more general concave functions $G$.  Motivated by a geometric construction of Joyce \cite{J, CP},  Donaldson established a correspondence between local solutions of (\ref{Lu}) and solutions of a certain second order linear equation.

In this paper, we study the  \emph{second boundary value problem} associated to the equation (\ref{Lu}).  Namely, let $G$ and $w$ be given by (\ref{G}) and (\ref{w}) for some $\theta \in [0,1/n)$.  We look for strictly convex functions $u$ satisfying
\begin{equation} \label{mainequation}
\begin{split}
L[u]:= {} & U^{ij} (w(d))_{ij} = f \quad \textrm{on } \Omega \\
u = {} & \varphi, \quad w(d) = \psi \quad \textrm{on } \partial \Omega,
\end{split}
\end{equation}
for given functions $\varphi$ and $\psi>0$ on $\ov{\Omega}$.  When $\theta=1/(n+2)$, this is the problem of prescribing the affine mean curvature of the graph of $u$ with Dirichlet boundary conditions on $u$ and the determinant of its Hessian $d$.

The problem (\ref{mainequation}) was introduced by Trudinger-Wang and  is an essential ingredient in their solution of the affine Plateau problem \cite{TW05}.  They showed that (\ref{mainequation}) admits solutions when $f \le 0$ (or if $f=f(x,u)$ and $f(x,t) \le 0$ for $t$ sufficiently negative).  Later, Trudinger-Wang used their results on boundary regularity for the Monge-Amp\`ere equation \cite{TW08} to obtain sharper results for $f \le 0$ with $f \in L^{\infty}(\Omega)$ or $f\in C^{\alpha}(\ov{\Omega})$.  Similar techniques give existence for $f \le \ve$ for small $\ve>0$  \cite{TW05, TW08}.  More recently, Le \cite{Le} dealt with the case of $f \le 0$ and $f \in L^p$ for $p>n$.  The case of general $f$, taking possibly large positive values, has until now remained open.

 Trudinger-Wang 
 \cite{TW05, TW08} noted that, when $n=1$, the second boundary value problem (\ref{mainequation}) does not admit solutions if $f$ is too large and positive.  Our main result states that (\ref{mainequation}) can be solved for general $f$ in every dimension except $n=1$.
 
More precisely, we prove the following.

\begin{theorem} \label{maintheorem} Assume $n \ge 2$.
Let $\Omega$ be a uniformly convex domain in $\mathbb{R}^n$ with $\partial \Omega \in C^{3, 1}$. Suppose $f \in L^{\infty}(\Omega)$, $\varphi \in C^{3, 1}(\overline{\Omega})$ and $0<\psi \in C^{1, 1}(\overline{\Omega})$.  Then there exists a unique uniformly convex solution $u \in W^{4, p}(\Omega)$ for all $1 < p < \infty$ to the second boundary value problem (\ref{mainequation}).

 If in addition, $\partial \Omega \in C^{4, \alpha}$ for some $\alpha \in (0,1)$, $f \in C^{\alpha}(\overline{\Omega})$, $\varphi \in C^{4, \alpha}(\overline{\Omega})$ and $0<\psi \in C^{2, \alpha}(\overline{\Omega})$ then $u \in C^{4, \alpha}(\ov{\Omega})$.
\end{theorem}

When $f \le 0$ the result of Theorem \ref{maintheorem} was proved by Trudinger-Wang \cite{TW08}, and our proof will make use of their work.

We remark that Theorem \ref{maintheorem} also holds for a more general function $G$.  We may assume that $G: (0,\infty) \rightarrow \mathbb{R}$ is a  smooth strictly concave function on $(0,\infty)$ whose derivative $w=G'$ is positive and satisfies:
\begin{enumerate}
\item[(A1)] $\displaystyle{w' + (1-\frac{1}{n}) \frac{w}{d} \le 0}$. 
\item[(A2)] $d w \ge c>0$ for some $c>0$ and all $d\ge 1$.
\item[(A3)] $\displaystyle{d^{1-1/n} w \rightarrow \infty}$ as $d \rightarrow 0$.
\end{enumerate}
We may also weaken the regularity assumptions on $f$.  Suppose we  
 replace the condition $f\in L^{\infty}(\Omega)$ by the two assumptions: 
 \begin{enumerate}
 \item[(a1)]  $\displaystyle{f \in L^{p}(\Omega)}$ for some $p>n$;
 \item[(a2)] $\| f^+ \|_{L^{\infty}(\Omega)} < \infty$.
  \end{enumerate}
 Then we obtain a solution $u \in W^{4,p}(\Omega)$.  
Furthermore, if $G$ is of the form $G(d) = \frac{d^{\theta}}{\theta}$ with $\theta \in (0,1/n)$ then we can replace condition (a2) by 
\begin{enumerate}
\item[(a2)*] $\int_{\Omega} (f^+)^q dx<\infty$ for some $q>1/\theta$,
\end{enumerate}
and still obtain a solution $u \in W^{4,p}(\Omega)$.  These results make use of an estimate of Le \cite{Le} who dealt with the case $f \in L^{p}(\Omega)$ ($p>n$) and $f \le 0$.  

\begin{remark} \emph{The  conditions (a2) or (a2)* may not be sharp.   It would be interesting to find the weakest regularity assumption on $f$ giving a solution $u \in W^{4,p}(\Omega)$.}
\end{remark}

\begin{remark} \emph{At least some conditions on $G$ are required to solve the second boundary value problem for arbitrary $f$. Indeed it was shown by Trudinger-Wang \cite{TW02} that if $n=2$ and $G$ is the \emph{convex} function $G(d) = d^2$, the equation $U^{ij}w_{ij} =f$ does not admit smooth solutions for positive $f$.  They constructed a radial function $u$, not $C^3$ smooth, with $U^{ij}w_{ij}$ constant and positive.}
\end{remark}

We give the proof of Theorem \ref{maintheorem} in Section \ref{section2}.   Since it takes no extra work, we will prove it for a general $G$ satisfying (A1)-(A3) above.  We will also deal with the cases of $f$ satisfying the weaker regularity assumptions described above (see Theorem \ref{thmgeneral} below).  

The key new estimate is the $\sup |u|$ bound, established in Lemma \ref{lemmau}.

In Section \ref{sectionproper}, we describe why the result of Theorem \ref{maintheorem}, and the nonexistence of solutions for some $f$ in the case $n=1$, are natural in the context of Monge-Amp\`ere functionals.  
Namely, we write down a functional associated to the second boundary value problem (\ref{mainequation}) and show that when $n\ge 2$ it satisfies a ``properness'' condition, which is somewhat analogous to the stability conditions of Donaldson \cite{D1}.  
When $n=1$, solutions to (\ref{mainequation}) hold if and only if we have properness.

Some remarks about notation.  We will use  $C$, $C'$ to denote uniform constants, which may differ from line to line, and whose uniformity will be clear from the context.  Integrals over $\Omega$ and $\partial \Omega$ will be taken with respect to $dx$ and the usual $(n-1)$-surface measure $ds_x$.

\section{Proof of Theorem \ref{maintheorem}}  \label{section2}

Let $G$ be a strictly concave function satisfying (A1), (A2) and (A3).  As in Theorem \ref{maintheorem}, assume that $n\ge 2$ and $\Omega$ is a uniformly convex domain.
We first prove  \emph{a priori} estimates for a solution $u$ of (\ref{mainequation}).

\begin{theorem} \label{thmapriori} 
The following estimates hold.
\begin{enumerate}
\item [(i)] Suppose $\partial \Omega \in C^{3,1}$, $f \in L^{\infty}(\Omega)$, $\varphi \in C^{3, 1}(\overline{\Omega})$ and $0<\psi \in C^{1, 1}(\overline{\Omega})$.  
If $u \in W^{4, p}(\Omega)$ is a strictly convex  solution of the second boundary value problem (\ref{mainequation}) with $p>n$, then 
$$\| u \|_{W^{4, p}(\Omega)}\le C, \quad \textrm{and } d \ge C^{-1}>0,$$
for a constant $C$ depending only on $n$, $p$, $\Omega$, the function $G$, $\| f\|_{L^{\infty}(\Omega)}$, $\| \varphi\|_{C^{3, 1}(\overline{\Omega})}$, $\| \psi\|_{C^{1, 1}(\ov{\Omega})}$ and $\inf_{\partial \Omega}\psi$.
\item [(ii)]  Suppose $\partial \Omega \in C^{4, \alpha}$ for some $\alpha \in (0,1)$, $f \in C^{\alpha}(\overline{\Omega})$, $\varphi \in C^{4, \alpha}(\overline{\Omega})$ and $0<\psi \in C^{2, \alpha}(\overline{\Omega})$. 
 If $u \in C^{4, \alpha}(\ov{\Omega})$ is a strictly convex  solution of the second boundary value problem (\ref{mainequation})  then 
$$\| u \|_{C^{4, \alpha}(\overline{\Omega})} \le C, \quad \textrm{and } d \ge C^{-1}>0,$$
for a constant $C$ depending only on $n$, $\alpha$, $\Omega$, the function $G$, $\| f\|_{C^{\alpha}(\ov{\Omega})}$, $\| \varphi\|_{C^{4, \alpha}(\overline{\Omega})}$, $\| \psi\|_{C^{2, \alpha}(\ov{\Omega})}$ and $\inf_{\partial \Omega}\psi$.
\end{enumerate}
\end{theorem}

In what follows, suppose that $u$ solves the second boundary value problem as in part (i) of Theorem \ref{thmapriori}. The key estimate is:

\begin{lemma}\label{lemmau} We have
$$\sup_{\Omega} |u| \le C.$$
where $C$ depends only on $n$, $p$, $\Omega$, the function $G$, $\| f\|_{L^1(\Omega)}$, $\| \varphi\|_{C^{3, 1}(\overline{\Omega})}$, $\| \psi\|_{C^{1, 1}(\ov{\Omega})}$ and $\inf_{\partial \Omega}\psi$.

\end{lemma}
\begin{proof}  Note that since $u$ is convex and equal to $\varphi$ on $\partial \Omega$, bounding $\sup_{\Omega} |u|$ is equivalent to bounding $\inf_{\Omega} u$ from below.

Let $\tilde{u} \in W^{4, p}(\Omega)$ be the Trudinger-Wang solution of (\ref{mainequation}) with $f=0$ and the same boundary data as $u$.  Namely, $\tilde{u}$ solves
$$L[\tilde{u}] = \tilde{U}^{ij} \tilde{w}_{ij} =0, \quad \tilde{u} = \varphi \quad \textrm{on } \partial \Omega, \quad \tilde{w} = \psi \quad \textrm{on } \partial \Omega.$$
Here we are using the obvious notation $\tilde{w}=G'(\tilde{d})$ for $\tilde{d}=\det D^2 \tilde{u}$, and $(\tilde{U}^{ij})$ for the cofactor matrix of $(\tilde{u}_{ij})$.  Note that by the estimates of Trudinger-Wang \cite{TW08} we have in particular the bounds
\begin{equation} \label{TWb}
\| \tilde{u} \|_{C^{2}(\ov{\Omega})} \le C, \quad \textrm{and } \tilde{d} \ge C^{-1}>0.
\end{equation}

Write $u_t = t\tilde{u} + (1-t) u$ for $t \in [0,1]$ and $A(t) = \int_{\Omega} G(\det D^2u_t)$.  By assumption (A1), the function $A(t)$ is concave (see for example \cite[Remark 2.1]{Zhou}).  Indeed, writing $\eta = \tilde{u} - u$ we have
\[
\begin{split}
A''(t) = {} & \int_{\Omega} \left( (w_t' + \frac{w_t}{d_t}) (U_t^{ij}\eta_{ij})^2  - \frac{w_t}{d_t} U_t^{ik} U_t^{j\ell} \eta_{k\ell} \eta_{ij} \right) \\
\le {} & \int_{\Omega} \left( w_t' + (1-\frac{1}{n} ) \frac{w_t}{d_t} \right) (U_t^{ij}\eta_{ij})^2 \le 0,
\end{split}
\]
where the quantities $w_t, w_t', d_t$ and $U_t^{ij}$ here are with respect to the convex function $u_t$.

  Integrating by parts twice and using the fact that $(U^{ij})_j=0$ we obtain
\begin{equation} \label{A10}
\begin{split}
A(1) -A(0) \le  A'(0) 
=  {} & \int_{\Omega} w U^{ij} (\tilde{u}_{ij} - u_{ij}) \\
= {} & \int_{\partial \Omega} \psi U^{ij} (\tilde{u}_j - u_j) \nu_i + \int_{\Omega} f(\tilde{u}-u),
\end{split}
\end{equation}
where $\nu_i$ are the components of the outward unit normal to $\partial \Omega$.  Similarly, using now the fact that $\tilde{U}^{ij}\tilde{w}_{ij}=0$,
\begin{equation} \label{A01}
\begin{split}
A(0)- A(1) \le - A'(1) 
=  {} & \int_{\Omega} \tilde{w} \tilde{U}^{ij} (u_{ij} - \tilde{u}_{ij}) \\
= {} & \int_{\partial \Omega} \psi \tilde{U}^{ij} (u_j - \tilde{u}_j) \nu_i.
\end{split}
\end{equation}

Adding these inequalities gives
\begin{equation} \label{key1}
\int_{\Omega} fu + \int_{\partial \Omega} \psi U^{ij} (u_j -\tilde{u}_j ) \nu_i  + \int_{\partial \Omega} \psi \tilde{U}^{ij} (\tilde{u}_j - u_j) \nu_i \le C.
\end{equation}
The argument for (\ref{key1}) is similar to Trudinger-Wang's proof of uniqueness  \cite[Lemma 7.1]{TW08}, except that here we take solutions of two \emph{different} equations.

We rewrite the integrals over $\partial \Omega$ as follows.  Given a point $p \in \partial \Omega$, choose coordinates $x_1, \ldots, x_n$ centered at $p$ so that the unit outward normal $\nu$ is in the negative $x_n$ direction.  Then, at $p$, since $u=\tilde{u}$ along $\partial \Omega$,
\begin{equation} \label{bdycalc0}
U^{ij} (u_j - \tilde{u}_j) \nu_i = U^{nn} (u_{\nu} - \tilde{u}_{\nu}), \quad \tilde{U}^{ij} (\tilde{u}_j - u_j) \nu_i = \tilde{U}^{nn} (\tilde{u}_{\nu} - u_{\nu}),
\end{equation}
where $u_{\nu} = - u_n$ is the derivative of $u$ in the direction of $\nu$.

We claim that at $p$,
\begin{equation} \label{bdycalc}
U^{nn}  = K (u_{\nu})^{n-1} + E, \quad \textrm{with } |E| \le C (1+ (u_{\nu})^{n-2}), 
\end{equation}
where $K$ denotes the Gauss curvature of $\partial \Omega$ at $p$.  

We now prove the claim, using an argument similar to that in \cite[Section 2]{CNS}. Since $\Omega$ is strictly convex, we may write  $\partial \Omega$ locally near $0$ as the graph $(x', \rho(x')) \in \mathbb{R}^n$ where $x'=(x_1, \ldots, x_{n-1})$ and
$$\rho(x') = \frac{1}{2} \sum_{\alpha, \beta=1}^{n-1} B_{\alpha \beta} x_{\alpha} x_{\beta} + O (|x'|^3),$$
for $(B_{\alpha \beta})$ a positive definite symmetric $(n-1) \times (n-1)$ matrix.  Note that the determinant of this matrix is precisely the Gauss curvature $K$ at $p$.
Since $u=\varphi$ on $\partial \Omega$ we have
$$(u-\varphi)(x', \rho(x')) =0$$
for $x'$ small.  
Hence for $\alpha \in \{1, \ldots, n-1 \}$ we have
$$ (\partial_{\alpha} + (\partial_{\alpha} \rho) \partial_n)(u-\varphi) =0,$$
for all small $x'$.  Differentiating with respect to $x_{\beta}$ for $\beta \in \{ 1, \ldots, n-1 \}$, we get
\begin{equation*}\label{star}
\qquad \qquad (\partial_{\beta} + (\partial_{\beta} \rho) \partial_n )(\partial_{\alpha} + (\partial_{\alpha} \rho) \partial_n)(u-\varphi) =0.
\end{equation*}
But $\partial_{\alpha} \rho = \sum_{\beta=1}^{n-1} B_{\alpha \beta} x_{\beta} +O(|x'|^2)$, and in particular it vanishes at $x=0$.  Hence at $x=0$ we obtain
\begin{equation*}\label{boundarycalculation}
u_{\alpha \beta}-\varphi_{\alpha \beta} +B_{\alpha \beta}  (u_n-\varphi_n) =0.\end{equation*}
We may rewrite this as
$$u_{\alpha \beta} = B_{\alpha \beta} u_{\nu} + \varphi_{\alpha \beta} -  B_{\alpha \beta} \varphi_{\nu}, \quad \alpha, \beta \in \{ 1, \ldots, n-1\}.$$
Taking determinants  proves the claim (\ref{bdycalc}).

We can now complete the proof of the lemma.  Since $u$ is convex and we wish to prove that $\inf_{\Omega} u$ is bounded we may assume that  $u_{\nu}$ is large at every point of $\partial \Omega$.  Indeed, the convexity of $u$  implies that
\begin{equation} \label{un}
u_{\nu}(x) \ge \frac{ \varphi(x)- \inf_{\Omega} u }{\textrm{diam}(\Omega)} - | D \varphi (x)|, \quad \textrm{for all } x\in \partial \Omega,
\end{equation}
provided the right hand side is nonnegative, which we may assume without loss of generality.   In particular, $Cu_{\nu} \ge \sup_{\Omega}|u|$.  To see (\ref{un}), let $V$ be an outward pointing unit vector at $x$ in the direction of the line segment between $x$ and the point in $\Omega$ at which $u$ achieves its minimum.  Then use the inequalities $u_{\nu} \langle V,  \nu \rangle \ge D_{V} u - | D \varphi|$ at $x$, and $0 <  \langle V,  \nu \rangle \le 1$.

On the other hand, $\tilde{u}_{\nu}$ and $\tilde{U}^{nn}$ are uniformly bounded  by (\ref{TWb}).  Hence we may assume that 
\begin{equation}
 \tilde{u}_{\nu} \le \frac{1}{4} u_{\nu}, \quad \tilde{U}^{nn} \le \frac{1}{4} U^{nn},
\end{equation}
where for the second inequality we have used (\ref{bdycalc}).  Thus we can assume 
\begin{equation} \label{sb}
U^{nn}(u_{\nu} - \tilde{u}_{\nu}) + \tilde{U}^{nn} (\tilde{u}_{\nu} - u_{\nu}) \ge \frac{1}{2} U^{nn} u_{\nu}.
\end{equation}
Combining (\ref{key1}), (\ref{bdycalc0}), (\ref{bdycalc}) and (\ref{sb}) gives
\begin{equation} \label{finally}
 \frac{1}{2} \int_{\partial \Omega} K \psi (u_{\nu})^n  \le C \int_{\partial \Omega} (u_{\nu})^{n-1} - \int_{\Omega} fu.
\end{equation} 
On the other hand, by (\ref{un}),
$$- \int_{\Omega} fu \le \| f\|_{L^1(\Omega)} \sup_{\Omega} |u| \le C \int_{\partial \Omega} u_{\nu}.$$

Since we may assume that $u_{\nu}$ is large compared to $2K^{-1} \psi^{-1}$, 
the inequality (\ref{finally}) implies the bound
$$\int_{\partial \Omega} (u_{\nu})^n \le C,$$
and hence by (\ref{un}) a uniform  bound for $\sup_{\Omega} |u|$.
\end{proof}

We can now prove estimates for $w$ using maximum principle arguments.  The  following lemma is contained in Trudinger-Wang \cite{TW05} and Le \cite{Le}, but we include the proof for the reader's convenience.

\begin{lemma} \label{lemmaw} There exists a uniform constant $p_0>0$ depending only on the constants of Lemma \ref{lemmau}, such that on $\Omega$,
\begin{enumerate}
\item[(i)] $C^{-1} \le w \le C$
\item[(ii)] $C^{-1} \le d \le C$,
\end{enumerate}
where $C$ depends only on $\|f^-\|_{L^n}$, $\| f^+\|_{L^{\infty}(\Omega)}$ and the constants in Lemma \ref{lemmau}.
\end{lemma}
\begin{proof}  
First we prove the upper bound of $w$.  Recall that $f^+ := \max (f,0)$ and $f^-:= \min(f,0)$.
We have 
$U^{ij}w_{ij}=f \ge f^-$ with $w|_{\partial\Omega}=\psi$.   Note that $\det U^{ij}=d^{n-1}$.  As in \cite[Lemma 3.1]{Le}, we apply Aleksandrov's maximum principle (see e.g. \cite[Theorem 9.1]{GT}) to give  \begin{equation}\label{Aleksandrov}\begin{split}\sup_{\Omega} w &\leq \sup_{\partial \Omega} \psi +C \left\|\frac{f^-}{d^{(n-1)/n}}\right\|_{L^n(\Omega)}\\
&\leq \sup_{\partial \Omega} \psi +C \left\| f^-\right \|_{L^n(\Omega)} \sup_{\Omega} (d^{(1-n)/n}),\\
\end{split}\end{equation}
where $C$ depends only on $n$ and $\Omega$.  The desired upper bound on $w$ follows from \eqref{Aleksandrov} and  assumption (A3) on $G$.  The lower bound for $d$ then follows immediately.

We next prove the lower bound of $w$ using the maximum principle, following Trudinger-Wang \cite{TW05}.  As there, we may assume that $w$ is in  $C^2$ by an approximation argument.  From the assumption (A2) on $G$ and the lower bound for $d$, we have $dw \ge c>0$ (after possibly shrinking the constant $c$).   Define
$$Q = \log w - Mu, \quad \textrm{for } M = \frac{\| f^+ \|_{L^{\infty}(\Omega)} +1}{nc}.$$  
We wish to bound $Q$ from below on $\ov{\Omega}$.  Suppose that $Q$ achieves a minimum at an interior point $p \in \Omega$.  Compute at $p$,
$$0 \le u^{ij} Q_{ij} = u^{ij} \left( \frac{w_{ij}}{w} - \frac{w_i w_j}{w^2} - M u_{ij} \right) \le \frac{f}{wd} - Mn \le \frac{f^+}{c} - Mn <0,$$
 a contradiction.  
  
 Hence $Q$ achieves its minimum at a point of the boundary, at which $w=\psi$.  It follows that $Q$ is bounded below, and from the bound on $\sup_{\Omega} |u|$ we obtain $w\ge C^{-1}>0$.
 
Finally, we give the upper bound for $d$. Observe that the assumption (A1) for $G$ implies that $(wd^{1-1/n})' \le 0$ and so
\begin{equation}\label{wd}
wd^{1-1/n} \le C, \quad \textrm{for } d\ge 1.
\end{equation}
The 
 upper bound for $d$ follows.
\end{proof}

Next, we include here a stronger bound for $w$ and $d$ (that is, depending on weaker norms of $f$) in the special case when $G =d^{\theta}/{\theta}$ for $\theta \in (0,1/n)$.

\begin{lemma} \label{lemmaw2} Let $G(d)=d^{\theta}/\theta$ for $\theta \in (0,1/n)$.  Then for any $q>\frac{1}{\theta}$, we have on $\Omega$,
\begin{enumerate}
\item[(i)] $C^{-1} \le w \le C$
\item[(ii)] $C^{-1} \le d \le C$,
\end{enumerate}
where $C>0$ depends only on $\theta$, $q$, $\|f^-\|_{L^n}$, $\int_{\Omega} (f^+)^q$  and the constants in Lemma \ref{lemmau}.
\end{lemma}
\begin{proof}
The only difference from the proof of Lemma \ref{lemmaw} is the lower bound for $w$.  Define
$$Q = \frac{1}{w} + Mu,$$
for $M$ a constant to be determined.  Then
\[
\begin{split}
u^{ij}Q_{ij} ={} &  - \frac{u^{ij} w_{ij}}{w^2} + 2 \frac{u^{ij} w_i w_j}{w^3} + Mn 
\\
\ge {} & -d^{1-2\theta} f + Mn \ge - \left( d^{1-2\theta} f - Mn\right)^+.
\end{split}
\]
Then Aleksandrov's maximum principle gives
\begin{equation*}
\begin{split}
\lefteqn{\frac{1}{\inf_{\Omega} w} - M \sup_{\Omega}|u|  \le  \sup_{\Omega} Q } \\ & \qquad \quad  \le  \sup_{\partial \Omega} \frac{1}{\psi} + M\sup_{\partial \Omega} \varphi + C (\sup_{\Omega} d)^{\frac{1}{n}+1-2\theta} \| (f - Mn d^{2\theta-1})^+ \|_{L^n(\Omega)}. 
\end{split}
\end{equation*}
Set $M = (\inf_{\Omega} w)^{-1}( 4 \sup_{\Omega}|u|+1)^{-1} = (\sup_{\Omega} d)^{1-\theta} ( 4 \sup_{\Omega}|u|+1)^{-1}$, so that
\begin{equation}  \label{1ow}
M \le C + C (\sup_{\Omega} d)^{\frac{1}{n}+1-2\theta} \| (f - Mn d^{2\theta-1})^+ \|_{L^n(\Omega)}.
\end{equation}

 Compute
\begin{equation} \label{1ow2}
\begin{split}
\lefteqn{\| (f - Mn d^{2\theta-1} )^+\|_{L^n(\Omega)} } \\ \le {} & \left( \int_{\{ f \ge Mnd^{2\theta-1} \} }(f^+)^n \right)^{1/n} \\
\le {} & \left( \int_{ \{ f \ge Mnd^{2\theta-1} \}} (f^+)^n \frac{ (f^+)^{q-n}} {(Mnd^{2\theta -1})^{q-n}} \right)^{1/n}.
\end{split}
\end{equation}
Combining (\ref{1ow}) and (\ref{1ow2}),
\[
\begin{split}
M \le {} & C (\sup_{\Omega} d)^{\frac{1}{n}(1-2q\theta +q)} M^{1-\frac{q}{n}}  
 \left( \int_{\Omega} (f^+)^q \right)^{1/n} \le C' M^{1- \frac{(q\theta-1)}{n(1-\theta)}}. \\
\end{split}
\]
Since $ q\theta>1$ we obtain an upper bound for $M$ and hence a lower bound for $w$.
\end{proof}

\begin{remark} \emph{
In the case $G(d)=\log d$, by considering the function $Q=-\log w+Mu$ and an argument similar to that above, we may bound $w$ from below provided
 $e^f$ lies in  $L^{p_0}(\Omega)$, although $p_0$ will depend on $\| f\|_{L^1(\Omega)}$.}
\end{remark}

It is now straightforward to finish the proof of Theorem \ref{thmapriori} by the arguments of Trudinger-Wang.

\begin{proof}[Proof of Theorem \ref{thmapriori}]  The proof is contained in the arguments of Lemma 7.3 and Lemma 7.4 of \cite{TW08}.  Nevertheless, we sketch here the basic ideas for the sake of completeness.  When $f\in L^{\infty}$ it follows from Lemma \ref{lemmaw}, the bound on $f$ and a simple barrier argument that $$|w(x) - w(x_0)| \le C|x-x_0|, \quad \textrm{for } x\in \Omega, \ x_0 \in \partial \Omega.$$
Moreover, since $w$ solves the linearized Monge-Amp\`ere equation 
\begin{equation} \label{lma}
U^{ij} w_{ij} =f,
\end{equation}
with $0< C^{-1} \le \det D^2 u \le C$, we can apply the arguments of   \cite[Lemma 7.3]{TW08} to obtain a uniform $C^{\beta}(\ov{\Omega})$ bound for $w$ for some $\beta \in (0,1)$.   This makes use of results of Caffarelli and Caffarelli-Guti\'errez \cite{Caf1, Caf2, CafGut}.

 Hence in both cases (i) and (ii), $\det D^2u$ is bounded in $C^{\beta}(\ov{\Omega})$ and the main theorem of \cite{TW08} implies that $u$ is bounded in $C^{2, \beta}(\ov{\Omega})$.  In particular, the functions $U^{ij}$ are bounded in $C^{\beta}(\ov{\Omega})$ and then standard elliptic estimates for the equation (\ref{lma}) imply that $w$ and hence $\det D^2u$ is bounded in $W^{2,p}(\Omega)$ for all $p> 1$ in case (i), and in $C^{2,\alpha}(\ov{\Omega})$ in case (ii).  The result follows. 
 \end{proof}

We can now complete the proof of Theorem \ref{maintheorem} using the Leray-Schauder degree theory argument of Trudinger-Wang \cite{TW05, TW08}.   Let $\Omega, \varphi, \psi, f$ be as in the first part of Theorem \ref{maintheorem}.

Fix $\alpha \in (0,1)$. For  a large constant $R>1$ to be determined, define a bounded set $D(R)$ in $C^{\alpha}(\ov{\Omega})$ as follows:
$$D(R) = \{ v \in C^{\alpha}(\ov{\Omega}) \ | \ v \ge R^{-1}, \ \| v\|_{C^{\alpha}(\ov{\Omega})} \le R\}.$$
Next, let $\Theta: (0,\infty) \rightarrow (0,\infty)$ be the inverse function of $w: (0,\infty) \rightarrow (0,\infty).$  

For $t \in [0,1]$, we will define an operator $\Phi_t : D(R) \rightarrow C^{\alpha}(\ov{\Omega})$ as follows.
Given $w \in D(R)$, define $u \in C^{2, \alpha}(\ov{\Omega})$ to be the unique strictly convex solution \cite{TW08} to
\begin{equation} \label{LS1}
\det(D^2 u) = \Theta(w) \quad \textrm{on } \Omega, \quad u = \varphi \quad \textrm{on } \partial \Omega.
\end{equation}
Next, let $w_t \in W^{2,p}(\Omega)$ (for some fixed $p>n$) be the unique solution to the equation
\begin{equation} \label{LS2}
U^{ij} (w_t)_{ij} = tf \quad \textrm{on } \Omega, \quad w_t=t\psi + (1-t) \quad \textrm{on } \partial \Omega.
\end{equation}
In particular, $w_t$ lies in  $C^{ \alpha}(\ov{\Omega})$.  We define $\Phi_t$ to be the map sending $w$ to $w_t$.

We note that:
\begin{enumerate}
\item[(i)] $\Phi_0(D(R)) = \{ 1\} $, and in particular, $\Phi_0$ has a unique fixed point.
\item[(ii)] The map $[0,1] \times D(R) \rightarrow C^{\alpha}(\ov{\Omega})$ given by $(t,w) \mapsto \Phi_t(w)$ is continuous.  
\item[(iii)] $\Phi_t$ is compact for each $t \in [0,1]$.
\item[(iv)] For every $t\in [0,1]$, if $w \in D(R)$ is a fixed point of $\Phi_t$ then $w \notin \partial D(R)$.
\end{enumerate}
Indeed, part (iii) follows from the standard \emph{a priori} estimates for the two separate equations (\ref{LS1}) and (\ref{LS2}).  For part (iv), let $w>0$ be a fixed point of $\Phi_t$.  Then $w \in W^{2,p}(\Omega)$ for some fixed $p>n$ and hence $u \in W^{4,p}(\Omega)$.  Next we apply Theorem \ref{thmapriori} to obtain $w>R^{-1}$ and $\| w \|_{C^{\alpha}(\ov{\Omega})} < R$ for some $R$ sufficiently large and depending only on the initial data.

Then the Leray-Schauder degree of $\Phi_t$ is well-defined for each $t$ and is constant on $[0,1]$ (see \cite[Theorem 2.2.4]{OCC}, for example).  $\Phi_0$ has a fixed point and hence $\Phi_1$ must also have a fixed point $w$, giving rise to a solution $u$ of   the second boundary value problem (\ref{mainequation}).
In the first case of Theorem \ref{maintheorem}, by the arguments above, the solution $u$ will lie in $W^{4,p}(\Omega)$ for all $p > 1$ and in the second case of Theorem \ref{maintheorem}, $u$ will lie in $C^{4, \alpha}(\ov{\Omega})$.

Note that the solution is uniformly convex since $\det D^2u \ge C^{-1}>0$.  The uniqueness follows from the same argument as in \cite[Lemma 7.1]{TW08}.  This completes the proof of Theorem \ref{maintheorem}.

We end this section by stating a more general version of the main theorem, which follows from our estimates together with an argument of Le \cite{Le}.

\begin{theorem} \label{thmgeneral}
Assume $n \ge 2$ and fix $p>n$.  Let $\Omega$ be a uniformly convex domain in $\mathbb{R}^n$ with $\partial \Omega \in C^{3, 1}$. Suppose $f \in L^{p}(\Omega)$, $\varphi \in W^{4,p}(\Omega)$ and $0<\psi \in W^{2,p}(\Omega)$.    Then the following hold:
\begin{enumerate}
\item[(a)]  Suppose that $G:(0,\infty) \rightarrow \mathbb{R}$ is a smooth strictly concave function satisfying (A1), (A2) and (A3). 
If $\| f^+ \|_{L^{\infty}(\Omega)} < \infty$ 
then there exists  a unique uniformly convex solution $u \in W^{4, p}(\Omega)$  to the second boundary value problem (\ref{mainequation}).
\item[(b)]  Suppose that $G(d) = d^\theta/\theta$ for some $\theta \in (0,1/n)$.  If 
$$\int_{\Omega} (f^+)^q < \infty$$
for some $q>1/\theta$ 
then there exists  a unique uniformly convex solution $u \in W^{4, p}(\Omega)$ to the second boundary value problem (\ref{mainequation}).
\end{enumerate}
\end{theorem}
\begin{remark} In part (b), the case when $f\leq 0$ was proved by Le \cite{Le}. 
\end{remark}

\begin{proof}[Proof of Theorem \ref{thmgeneral}]
We only need to make a couple of changes compared to the proof of Theorem \ref{maintheorem}.  The \emph{a priori} estimates for $w$ and $d$ for parts (a) and (b) follow from Lemmas \ref{lemmaw} and \ref{lemmaw2} respectively.
To obtain the $C^{\beta}(\ov{\Omega})$ bound for $w$, we apply Theorem 1.4 of \cite{Le}.  The rest of the arguments follow in the same way.
\end{proof}

\section{Monge-Amp\`ere functionals} \label{sectionproper}

In this section we discuss a Monge-Amp\`ere functional whose Euler-Lagrange equation is (\ref{mainequation}).  We describe how the form of this functional suggests, at least philosophically, that the result of Theorem \ref{maintheorem} should hold.  We also show that in dimension $n=1$, existence of solutions to the second boundary value problem is equivalent to a notion of properness for this functional.

For simplicity, we assume in this section that the boundary data $\varphi$ for $u$ is zero. 
Given  $\psi>0$ on $\partial \Omega$, define
$$\mathcal{S} = \{ u \in C^2(\ov{\Omega})  \ \textrm{strictly convex on $\Omega$ } | \ u|_{\partial \Omega}=0, \ w|_{\partial \Omega}=\psi \},$$
for $w=G'(d)$.
Given a function $f$ on $\Omega$, we define a functional $\mathcal{F}: \mathcal{S} \rightarrow \mathbb{R}$ by
$$\mathcal{F}(u) = \int_{\Omega} G(d) - \int_{\Omega} uf - \frac{1}{n} \int_{\partial \Omega} K \psi \, u_{\nu}^n,$$
where $K$ is the Gauss curvature of $\partial \Omega$ and $u_{\nu}$ the outward-facing normal derivative of $u$.  In the case of $n=1$ we take $K=1$.  Note that if we have nonzero boundary data for $u$ then we need to add a lower order term, of order $O(u_{\nu}^{n-1})$, to $\mathcal{F}$. 

The next proposition shows that the equation $U^{ij}w_{ij}=f$ is the Euler-Lagrange equation for $\mathcal{F}$.  Note that the result holds for any smooth $G$.

\begin{proposition}  \label{propEL} Let $u_t$ be a smooth path in $\mathcal{S}$ with $\frac{\partial}{\partial t} u_t|_{t=0}=\eta$.  Then
$$\frac{\partial}{\partial t}\bigg|_{t=0} \mathcal{F}(u_t) = \int_{\Omega} (U^{ij} w_{ij} - f) \eta.$$
\end{proposition}
\begin{proof}
Compute
\[
\begin{split}
\frac{\partial}{\partial t}\bigg|_{t=0} \mathcal{F}(u_t) = {} & \int_{\Omega} w U^{ij} \eta_{ij} - \int_{\Omega} \eta f - \int_{\partial \Omega} K \psi u_{\nu}^{n-1} \eta_{\nu} \\
= {} & - \int_{\Omega} w_i U^{ij} \eta_j + \int_{\partial \Omega} \psi U^{ij} \eta_i \nu_j - \int_{\Omega} \eta f - \int_{\partial \Omega} K \psi u_{\nu}^{n-1} \eta_{\nu}  \\
= {} & \int_{\Omega} (U^{ij} w_{ij} - f) \eta +\int_{\partial \Omega} \psi U^{ij} \eta_i \nu_j - \int_{\partial \Omega} K \psi u_{\nu}^{n-1}  \eta_{\nu},
\end{split}
\]
where for the last line we have used the fact that $\eta$ vanishes on $\partial \Omega$.  But since $u=0$ on $\partial \Omega$, we can apply the argument of (\ref{bdycalc}) to see that on $\partial \Omega$, we have $U^{ij} \eta_{i} \nu_j = K (u_{\nu})^{n-1} \eta_{\nu}$, as required.
\end{proof}

In \cite{D1}, Donaldson considers  a functional of the form
$$u \mapsto \int_{\Omega} G(d) - \mathcal{L}(u),$$
with $G(d)=\log d$ and $\mathcal{L}(u)$ a lower order functional, which in his case is linear.  
In our case, the functional $\mathcal{L}: \mathcal{S} \rightarrow \mathbb{R}$ is given by
\begin{equation} \label{L}
\mathcal{L}(u)=  \int_{\Omega} uf + \frac{1}{n} \int_{\partial \Omega} K \psi \, u_{\nu}^n.
\end{equation}
We make a definition here that $\mathcal{L}$ is \emph{proper} if, for some $\lambda>0$ and constant $C$,
\begin{equation} \label{properness}
\mathcal{L}(v) \ge \lambda \int_{\partial\Omega} v_{\nu} -C, \quad \textrm{for all } v \in \mathcal{S}.
\end{equation}  This condition is reminiscent of Donaldson's ``stability'' condition in \cite[Condition 1]{D1}.  A key idea in \cite{D1} is that  ``stability'' or ``properness'' of the $\mathcal{L}$-functional should be equivalent to the solvability of the corresponding Euler-Lagrange equation.  Similar ideas were later explored by Le-Savin \cite{LS} with different boundary conditions and  a more general function $G$.  Also related to this is the equivalence of the existence of K\"ahler-Einstein metrics and a Moser-Trudinger type inequality \cite{T, TZ, PSSW}.

A key point we wish to make is that if  $n\ge 2$, the functional $\mathcal{L}$ given by (\ref{L}) is \emph{always} proper.  Indeed, this follows from (\ref{un}) and  the fact that  $K\psi$ is uniformly bounded from below on $\partial \Omega$.  This immediately gives a strong reason to expect that one can always solve the second boundary value problem, as in the statement of our main result, Theorem \ref{maintheorem}.

On the other hand, the properness of $\mathcal{L}$ does not always hold when $n=1$, as can be seen by taking $f$ sufficiently large.  We can prove in this case that properness of $\mathcal{L}$ is indeed equivalent to solvability of the second boundary value problem.

\begin{proposition} \label{prop2}  Suppose that $n=1$, $G(d)= d^{\theta}/\theta$ for some $\theta\in [0,1)$, $\Omega =(a,b) \subset \mathbb{R}$, $f \in C^{\alpha}(\ov{\Omega})$ for some $\alpha \in (0,1)$  and $\psi(a), \psi(b)$ are positive real numbers. Then there exists a strictly convex $u \in C^{4,\alpha}(\ov{\Omega})$ solving the second boundary value problem
\begin{equation} \label{d1}
U^{ij} w_{ij} = f, \quad u|_{\partial \Omega} =0, \quad w|_{\partial \Omega} = \psi
\end{equation}
if and only if $\mathcal{L}$ defined by (\ref{L}) is proper.
\end{proposition}
\begin{proof}
In dimension $1$, the equation (\ref{d1}) simplifies considerably, since $d=\det D^2 u =u''$ and $U^{ij} w_{ij}=f$ becomes $w'' =f$.

First assume that $u$ solves (\ref{d1}).  Then writing $w=w(u'')$, we have for every $v\in \mathcal{S}$,
$$\int_{\Omega} fv =  \int_{\Omega} w'' v 
 = - \int_{\Omega} w'  v' 
 =  \int_{\Omega}  w v'' - \int_{\partial \Omega} \psi v_{\nu}. $$
Hence for $\lambda = \inf_{\Omega} w>0$, we obtain
$$\int_{\Omega} fv+  \int_{\partial \Omega} \psi v_{\nu} \ge \lambda \int_{\Omega} v'' = \lambda \int_{\partial \Omega} v_{\nu},$$
as required.

Conversely, suppose that $\mathcal{L}$ is proper on $\mathcal{S}$.  Then (\ref{properness}) holds for some $\lambda>0$. To show the existence of solutions to (\ref{d1}) it suffices to obtain an \emph{a priori} estimate on $\sup_{\Omega} |u|$.  Indeed, we can find a solution to $w''=f$ by integrating, and the only thing we need to check is that $w$ has an \emph{a priori} lower bound away from zero.  But this is  obtained by applying the maximum principle to the quantity $-\log w +Mu$ for sufficiently large $M$ (here we use the fact that $G=d^{\theta}/\theta$ for $\theta \in [0,1)$).

Let $\tilde{u}$ be a strictly convex function solving 
\begin{equation} \label{letutilde}
\tilde{w}''=0, \quad \tilde{u}|_{\partial \Omega} = 0, \quad \tilde{w}|_{\partial \Omega} = \frac{\lambda}{2}, 
\end{equation}
where we are writing $\tilde{w}$ for $w(\tilde{u}'')$.  By the result of \cite{TW05}, we know that such a $\tilde{u}$ exists, but in this case, one could even write down the solution explicitly.

By the same arguments as in (\ref{A10}),  (\ref{A01}) and (\ref{key1}) but noting that here $\tilde{w}$ has different boundary values from $w$, and $U^{ij}=\tilde{U}^{ij}=1$, we have
\begin{equation} \label{kfc}
\int_{\Omega} fu + \int_{\partial \Omega} \psi u_{\nu} \le \frac{\lambda}{2} \int_{\partial \Omega} u_{\nu} + C,
\end{equation}
for $C$ depending on $\psi$, $\lambda$, $\Omega$ and bounds for $f$.

Hence by the properness assumption,
$$\frac{\lambda}{2} \int_{\partial \Omega}  u_{\nu} \le C',$$
which implies an upper bound for  $\sup_{\Omega} |u|$, as required.
\end{proof}

Although we have stated Proposition \ref{prop2} for zero boundary data, the same proof works for general boundary data $\varphi$.

Of course the 
 equation (\ref{d1})  for $n=1$ is a very simple ODE which is  itself not of particular interest to us.    The inclusion of Proposition \ref{prop2} is to illustrate a general relationship between solutions of fourth order equations and behavior of the corresponding functionals.

\end{document}